\documentclass[11pt]{article}

\usepackage{epsfig,epsf,url,fancybox}
\usepackage{amsmath}
\usepackage{mathrsfs}
\usepackage{amssymb}
\usepackage{amsthm}
\usepackage{graphicx}
\usepackage{color}
\usepackage{hyperref}
\usepackage{algorithm}
\usepackage{algorithmic}
\usepackage{subcaption}
\usepackage[left=1in,top=1in,right=1in,bottom=1in,a4paper]{geometry}

\numberwithin{equation}{section}

\newtheorem{theorem}{Theorem}[section]
\newtheorem{corollary}[theorem]{Corollary}
\newtheorem{lemma}[theorem]{Lemma}

\newtheorem{assumption}[theorem]{Assumption}
\newtheorem{remark}[theorem]{Remark}

\newcommand{\T}{\top}
\newcommand{\E}{\mathbb{E}}
\newcommand\numberthis{\addtocounter{equation}{1}\tag{\theequation}}
\newcommand{\etal}{ et al. }
\newcommand{\argmin}{\mathop{\rm argmin}}
\newcommand{\br}{\mathbb{R}}

\newcommand{\be}{\begin{equation}}
\newcommand{\ee}{\end{equation}}
\newcommand{\ba}{\begin{array}}
\newcommand{\ea}{\end{array}}

\begin{document}

\title{Barzilai-Borwein Step Size for Stochastic Gradient Descent}

\author{Conghui Tan\footnotemark[2]\qquad Shiqian Ma\footnotemark[2]\qquad Yu-Hong Dai\footnotemark[3]\qquad Yuqiu Qian\footnotemark[4]}

\date{May 16, 2016}

\maketitle

\renewcommand{\thefootnote}{\fnsymbol{footnote}}


\footnotetext[2]{Department of Systems Engineering and Engineering Management, The Chinese University of Hong Kong, Hong Kong. Email: chtan@se.cuhk.edu.hk, sqma@se.cuhk.edu.hk.}

\footnotetext[3]{Academy of Mathematics and Systems Science, Chinese Academy of Sciences, Beijing, China. Email: dyh@lsec.cc.ac.cn.}

\footnotetext[4]{Department of Computer Science, The University of Hong Kong, Hong Kong. Email: qyq79@connect.hku.hk.}

\renewcommand{\thefootnote}{\arabic{footnote}}

\begin{abstract}
One of the major issues in stochastic gradient descent (SGD) methods is how to choose an appropriate step size while running the algorithm. Since the traditional line search technique does not apply for stochastic optimization algorithms, the common practice in SGD is either to use a diminishing step size, or to tune a fixed step size by hand, {which can be time consuming in practice.}
In this paper, we propose to use the Barzilai-Borwein (BB) method to automatically compute step sizes for SGD and its variant: stochastic variance reduced gradient (SVRG) method, which leads to two algorithms: SGD-BB and SVRG-BB. We prove that SVRG-BB converges linearly for strongly convex objective functions. As a by-product, we prove the linear convergence result of SVRG with {\it Option I} proposed in \cite{svrg}, whose convergence result is missing in the literature. Numerical experiments on standard data sets show that the performance of SGD-BB and SVRG-BB is comparable to and sometimes even better than SGD and SVRG with best-tuned step sizes, and is superior to some advanced SGD variants.
\end{abstract}

\section{Introduction}
The following optimization problem, which minimizes the sum of cost functions over samples from a finite training set, appears frequently in machine learning:
\begin{equation}\label{eq:problem}
    \min \ F(x)\equiv\frac{1}{n}\sum_{i=1}^n f_i(x),
\end{equation}
where $n$ is the sample size, and each $f_i:\br^d\rightarrow\br$ is the cost function corresponding to the $i$-th sample data. Throughout this paper, we assume that each $f_i$ is convex and differentiable, and the function $F$ is strongly convex.
Problem \eqref{eq:problem} is challenging when $n$ is extremely large so that computing $F(x)$ and $\nabla F(x)$ for given $x$ is prohibited. Stochastic gradient descent (SGD) method and its variants have been the main approaches for solving \eqref{eq:problem}.
In the $t$-th iteration of SGD, a random training sample $i_t$ is chosen from $\{1,2,\dots,n\}$ and the iterate $x_t$ is updated by
\begin{equation}\label{sgd}
    x_{t+1}= x_{t}-\eta_t \nabla f_{i_t}(x_t),
\end{equation}
where $\nabla f_{i_t}(x_t)$ denotes the gradient of the $i_t$-th component function at $x_t$, and $\eta_t>0$ is the step size (a.k.a.  learning rate). In \eqref{sgd}, it is usually assumed that $\nabla f_{i_t}$ is an unbiased estimation to $\nabla F$, i.e.,
\be\label{unbiased}\E[\nabla f_{i_t}(x_t)\mid x_t] = \nabla F(x_t).\ee
However, it is known that the total number of gradient evaluations of SGD depends on the variance of the stochastic gradients and it is of sublinear convergence rate for the strongly convex and smooth problem \eqref{eq:problem}, which is inferior to the full gradient method. As a result, many works along this line have been focusing on designing variants of SGD that can reduce the variance and improve the complexity. Some popular methods along this line are briefly summarized as follows. The stochastic average gradient (SAG) method proposed by Le Roux \etal \cite{schmidt2013minimizing} updates the iterates by
\be\label{sag}
    x_{t+1} = x_t - \frac{\eta_t}{n}\sum_{i=1}^n y_i^t,
\ee
where at each iteration a random training sample $i_t$ is chosen and $y_i^t$ is defined as
\[y_i^t = \left\{\ba{ll}\nabla f_{i}(x_t) & \mbox{ if } i=i_t, \\ y_i^{t-1}, & \mbox{ otherwise.}\ea\right.\]
It is shown in \cite{schmidt2013minimizing} that SAG converges linearly for strongly convex problems. The SAGA method proposed by Defazio \etal \cite{defazio2014saga} is an improved version of SAG, and it does not require the strong convexity assumption. 
{It is noted that SAG and SAGA need to store the latest gradients for the $n$ component functions $f_i$. }
The SDCA method proposed by Shalev-Shwartz and Zhang \cite{Shai-Zhang-SDCA} also requires to store all the component gradients.
The stochastic variance reduced gradient (SVRG) method proposed by Johnson and Zhang \cite{svrg} is now widely used in the machine learning community for solving \eqref{eq:problem}, because it achieves the variance reduction effect for SGD, and it does not need to store the $n$ component gradients.

As pointed out by Le Roux \etal \cite{schmidt2013minimizing}, one important issue regarding to stochastic algorithms (SGD and its variants) that has not been fully addressed in the literature, is how to choose an appropriate step size $\eta_t$ while running the algorithm. In classical gradient descent method, the step size is usually obtained by employing line search techniques. However, line search is computationally prohibited in stochastic gradient methods because one only has sub-sampled information of function value and gradient. As a result, for SGD and its variants used in practice, people usually use a diminishing step size $\eta_t$, or use a best-tuned fixed step size. Neither of these two approaches can be efficient. 

Some recent works that discuss the choice of step size in SGD are summarized as follows. AdaGrad \cite{duchi2011adaptive} scales the gradient by the square root of the accumulated magnitudes of the gradients in the past iterations, but it still requires a fixed step size $\eta$. {\cite{schmidt2013minimizing} suggests a line search technique on the component function $f_{i_k}(x)$ selected in each iteration, to estimate step size for SAG.} 
\cite{mahsereci2015probabilistic} suggests performing line search for an estimated function, which is evaluated by a Gaussian process with samples $f_{i_t}(x_t)$. \cite{masse2015speed} suggests to generate the step sizes by a given function with an unknown parameter, and to use the online SGD to update this unknown parameter.

{\bf Our contributions} in this paper are in several folds.
\begin{itemize}
\item[(i).] We propose to use the Barzilai-Borwein (BB) method to compute the step size for SGD and SVRG. The two new methods are named as SGD-BB and SVRG-BB, respectively. The per-iteration computational cost of SGD-BB and SVRG-BB is almost the same as SGD and SVRG, respectively.
\item[(ii).] We prove the linear convergence of SVRG-BB for strongly convex functions. As a by-product, we show the linear convergence of SVRG with Option I (SVRG-I) proposed in \cite{svrg}. Note that in \cite{svrg} only convergence of SVRG with Option II (SVRG-II) was given, and the proof for SVRG-I has been missing in the literature. However, SVRG-I is numerically a better choice than SVRG-II, as demonstrated in \cite{svrg}.
\item[(iii).] We conduct numerical experiments for SGD-BB and SVRG-BB on solving logistic regression and SVM problems. The numerical results show that SGD-BB and SVRG-BB are comparable to and sometimes even better than SGD and SVRG with best-tuned step sizes. We also compare SGD-BB with some advanced SGD variants, and demonstrate that our method is superior.
\end{itemize}

The rest of this paper is organized as follows. In Section \ref{sec:bb} we briefly introduce the BB method in the deterministic setting. In Section \ref{sec:svrg-bb} we propose our SVRG-BB method, and prove its linear convergence for strongly convex functions. As a by-product, we also prove the linear convergence of SVRG-I. In Section \ref{sec:sgd-bb} we propose our SGD-BB method. A smoothing technique is also implemented to improve the performance of SGD-BB. We conduct numerical experiments for SVRG-BB and SGD-BB in Section \ref{sec:num}. Finally, we draw some conclusions in Section \ref{sec:con}.

\section{The Barzilai-Borwein Step Size}\label{sec:bb}
The BB method, proposed by Barzilai and Borwein in \cite{barzilai1988two}, has been proved to be very successful in solving nonlinear optimization problems. The key idea behind the BB method is motivated by quasi-Newton methods. Suppose we want to solve the unconstrained minimization problem
\be\label{unconstrained}\min_x \ f(x),\ee
where $f$ is differentiable. A typical iteration of quasi-Newton methods for solving \eqref{unconstrained} takes the following form:
\begin{equation}\label{eq:quasi-newton}
    x_{t+1} = x_t - B_t^{-1}\nabla f(x_t),
\end{equation}
where $B_t$ is an approximation of the Hessian matrix of $f$ at the current iterate $x_t$. Different choices of $B_t$ give different quasi-Newton methods. The most important feature of $B_t$ is that it must satisfy the so-called secant equation:
\begin{equation}\label{eq:secant_equation}
    B_t s_t=y_t,
\end{equation}
where $s_t=x_t-x_{t-1}$ and $y_t=\nabla f(x_t)-\nabla f(x_{t-1})$ for $t\geq 1$.
It is noted that in \eqref{eq:quasi-newton} one needs to solve a linear system, which may be time consuming when $B_t$ is large and dense.
One way to alleviate this burden is to use the BB method, which replaces $B_t$ by a scalar matrix $\frac{1}{\eta_t} I$. However, one cannot choose a scalar $\eta_t$ such that the secant equation \eqref{eq:secant_equation} holds with $B_t=\frac{1}{\eta_t} I$. Instead, one can find $\eta_t$ such that the residual of the secant equation is minimized, i.e.,
\[\min_{\eta_t} \left\| \frac{1}{\eta_t} s_t - y_t \right\|_2^2,\]
which leads to the following choice of $\eta_t$:
\be\label{bb1-eta} \eta_t = \frac{\|s_t\|_2^2}{s_t^\T y_t}. \ee
Therefore, a typical iteration of the BB method for solving \eqref{unconstrained} is
\be\label{bb-update} x_{t+1} = x_t - \eta_t\nabla f(x_t),\ee
where $\eta_t$ is computed by \eqref{bb1-eta}.

\begin{remark}
Another choice of $\eta_t$ is obtained by solving
\[\min_{\eta_t} \| s_t - \eta_t y_t \|_2^2,\]
which leads to
\be\label{bb2-eta}\eta_t = \frac{s_t^\T y_t}{\|y_t\|_2^2}. \ee
In this paper, we will focus on the choice in \eqref{bb1-eta}, because the practical performance of \eqref{bb1-eta} and \eqref{bb2-eta} are similar.
\end{remark}

For convergence analysis, generalizations and variants of the BB method, we refer the interested readers to \cite{raydan1993barzilai,raydan1997barzilai,fletcher2005barzilai,Dai-Fletcher-2005,Dai-CBB-2006,Dai-2013} and references therein.
Recently, BB method has been successfully applied for solving problems arising from emerging applications, such as compressed sensing \cite{sparsa-2009}, sparse reconstruction \cite{Wen-SISC-2010} and image processing \cite{Wang-Ma-2007}.


\section{Barzilai-Borwein Step Size for SVRG}\label{sec:svrg-bb}
We see from \eqref{bb-update} and \eqref{bb1-eta} that the BB method does not need any parameter and the step size is computed while running the algorithm. This has been the main motivation for us to work out a black-box stochastic gradient descent method that can compute the step size automatically without requiring any parameters. In this section, we propose to incorporate the BB step size to SVRG which leads to the SVRG-BB method.

The following assumption is made throughout this section.
\begin{assumption}\label{assump-1}
We assume that \eqref{unbiased} holds for any $x_t$. We assume that the objective function $F(x)$ is $\mu$-strongly convex, i.e.,
\begin{equation*}
    F(y)\geq F(x)+\nabla F(x)^\T(y-x)+\frac{\mu}{2}\|x-y\|^2_2,\quad \forall x,y \in \mathbb{R}^d.
\end{equation*}
We also assume that the gradient of each component function $f_i(x)$ is $L$-Lipschitz continuous, i.e.,
\begin{equation*}
    \|\nabla f_i(x)-\nabla f_i(y)\|_2\leq L\|x-y\|_2,\; \forall x,y \in \mathbb{R}^d.
\end{equation*}
Under this assumption, it is easy to see that $\nabla F(x)$ is also $L$-Lipschitz continuous:
\begin{equation*}
    \|\nabla F(x)-\nabla F(y)\|_2\leq L\|x-y\|_2,\; \forall x,y \in \mathbb{R}^d.
\end{equation*}
\end{assumption}

\subsection{SVRG Method}
The SVRG method proposed by Johnson and Zhang \cite{svrg} for solving \eqref{eq:problem} is described as in Algorithm \ref{alg:svrg}. 

\begin{algorithm}[ht]
    \caption{Stochastic Variance Reduced Gradient (SVRG) Method}\label{alg:svrg}
    \begin{algorithmic}
        \STATE \textbf{Parameters}: update frequency $m$, step size $\eta$, initial point $\tilde{x}_0$
        \FOR{$k=0, 1,\cdots$}
            \STATE $g_k=\frac{1}{n}\sum_{i=1}^n\nabla f_i(\tilde{x}_k)$
            \STATE $x_0 = \tilde{x}_k$
            \STATE $\eta_k = \eta$
            \FOR{$t=0,\cdots,m-1$}
                \STATE Randomly pick $i_t\in \{1,\dots,n\}$
                \STATE $x_{t+1}=x_t-\eta_k(\nabla f_{i_t}(x_t) - \nabla f_{i_t}(\tilde{x}_k) + g_k)$
            \ENDFOR
            \STATE \textbf{Option I}: $\tilde{x}_{k+1}=x_{m}$
            \STATE \textbf{Option II}: $\tilde{x}_{k+1}=x_{t}$ for randomly chosen $t\in \{1,\dots,m\}$
        \ENDFOR
   \end{algorithmic}
\end{algorithm}

There are two loops in SVRG (Algorithm \ref{alg:svrg}). In the outer loop (each outer iteration is called an {\it epoch}), a full gradient $g_k$ is computed, which is used in the inner loop for generating stochastic gradients with lower variance. $\tilde{x}$ is then chosen, based on the outputs of inner loop, for the next outer loop. Note that two options for choosing $\tilde{x}$ are suggested in SVRG. Intuitively, Option I in SVRG (denoted as SVRG-I) is a better choice than Option II (denoted as SVRG-II), because the former used the latest information from the inner loop. This has been confirmed numerically in \cite{svrg} where SVRG-I was applied to solve real applications. However, the convergence analysis is only available for SVRG-II (see, e.g., \cite{svrg}, \cite{konevcny2013semi} and \cite{harikandeh2015stopwasting}), and the convergence for SVRG-I has been missing in the literature. We now cite the convergence analysis of SVRG-II given in \cite{svrg} as follows.

\begin{theorem}[\cite{svrg}]\label{thm:svrg}
Consider SVRG in Algorithm \ref{alg:svrg} with {\it Optioin II}. Let $x^*$ be the optimal solution to problem \eqref{eq:problem}. Assume that $m$ is sufficiently large so that
\begin{equation}\label{convergence-svrg-ii-alpha}
\alpha:=\frac{1}{\mu\eta(1-2L\eta)m}+\frac{2L\eta}{1-2L\eta}<1,
\end{equation}
then we have linear convergence in expectation for SVRG:
\begin{equation*}
\E\left[F(\tilde{x}_k)-F(x^*)\right]\leq \alpha^k[F(\tilde{x}_0)-F(x^*)].
\end{equation*}
\end{theorem}

There has been a series of follow-up works on SVRG and its variants. Xiao and Zhang \cite{xiao2014proximal} developed a proximal SVRG method for minimizing the finite sum function plus a nonsmooth regularizer.
\cite{nitanda2014stochastic} applied Nesterov's acceleration technique to SVRG to improve the convergence rate that depends on the condition number $L/\mu$. \cite{harikandeh2015stopwasting} proved if the full gradient computation $g_k$ was replaced by a growing-batch estimation, the linear convergence rate can be preserved. \cite{allen2016variance} and \cite{Reddi2016SVRG} showed that SVRG with minor modifications can converge to a stationary point for nonconvex optimization problems.

\subsection{SVRG-BB Method}
It is noted that in SVRG, the step size $\eta$ needs to be provided by the user. According to \eqref{convergence-svrg-ii-alpha}, the choice of $\eta$ is dependent on $L$, which may be difficult to estimate in practice. In this section, we propose the SVRG-BB method that computes the step size using the BB method.
Our SVRG-BB algorithm is described in Algorithm \ref{alg:svrg-bb}. Note that the only difference between SVRG and SVRG-BB is that in the latter we use BB method to compute the step size $\eta_k$, instead of using a prefixed $\eta$ as in SVRG. 

\begin{algorithm}[ht]
    \caption{SVRG with BB step size (SVRG-BB)}\label{alg:svrg-bb}
    \begin{algorithmic}
        \STATE \textbf{Parameters}: update frequency $m$, initial point $\tilde{x}_0$, initial step size $\eta_0$ (only used in the first epoch)
        \FOR{$k=0, 1,\cdots$}
            \STATE $g_k=\frac{1}{n}\sum_{i=1}^n\nabla f_i(\tilde{x}_k)$
            \IF{$k>0$}
                \STATE \be\label{alg:svrg-bb-eta}\eta_k=\frac{1}{m}\cdot\|\tilde{x}_k-\tilde{x}_{k-1}\|^2_2/(\tilde{x}_k-\tilde{x}_{k-1})^\T(g_k-g_{k-1})\ee 
            \ENDIF
            \STATE $x_0 = \tilde{x}_k$
            \FOR{$t=0,\cdots,m-1$}
                \STATE Randomly pick $i_t\in \{1,\dots,n\}$
                \STATE $x_{t+1}=x_t-\eta_k(\nabla f_{i_t}(x_t) - \nabla f_{i_t}(\tilde{x}_k) + g_k)$
            \ENDFOR
            \STATE $\tilde{x}_{k+1}=x_{m}$
        \ENDFOR
   \end{algorithmic}
\end{algorithm}

\begin{remark}
A few remarks are in demand for the SVRG-BB algorithm.
\begin{enumerate}
\item One may notice that $\eta_k$ is equal to the step size computed by the BB formula \eqref{bb1-eta} divided by $m$. This is because in the inner loop for updating $x_t$, $m$ unbiased gradient estimators are added to $x_0$ to get $x_m$.
\item If we always set $\eta_k = \eta$ in SVRG-BB instead of using \eqref{alg:svrg-bb-eta}, then it reduces to SVRG-I.
\item For the first outer loop of SVRG-BB, a step size $\eta_0$ needs to be specified, because we are not able to compute the BB step size for the first outer loop. However, we observed from our numerical experiments that the performance of SVRG-BB is not sensitive to the choice of $\eta_0$.
\item The BB step size can also be naturally incorporated to other SVRG variants, such as SVRG with batching \cite{harikandeh2015stopwasting}.
\end{enumerate}
\end{remark}

\subsection{Linear Convergence Analysis}
\label{sec:proof}
In this section, we analyze the linear convergence of SVRG-BB (Algorithm \ref{alg:svrg-bb}) for solving \eqref{eq:problem} with strongly convex objective $F(x)$, and as a by-product, our analysis also proves the linear convergence of SVRG-I.

The following lemma, which is from \cite{nesterov2004introductory}, is useful in our analysis.
\begin{lemma}[co-coercivity] \label{co-coercivity}
If $f(x):\mathbb{R}^d\rightarrow \mathbb{R}$ is convex and its gradient is $L$-Lipschitz continuous, then
\begin{equation*}
    \|\nabla f(x)-\nabla f(y)\|^2_2\leq L( x-y)^\T (\nabla f(x) -\nabla f(y) ),\quad \forall x,y  \in \mathbb{R}^d.
\end{equation*}
\end{lemma}

In the following, we first prove the following lemma, which reveals the relationship between the distances of two consecutive iterates to the optimal point.

\begin{lemma}\label{thm:new-svrg}
Define
\begin{equation}\label{eq:def-alpha}
    \alpha_k:=\left(1-2\eta_k\mu(1-\eta_k L)\right)^m + \frac{4\eta_k L^2}{\mu(1-\eta_k L)}.
\end{equation}
For both SVRG-I and SVRG-BB, we have the following inequality for the $k$-th epoch:
\begin{equation*}
    \E\left\|\tilde{x}_{k+1} - x^*\right\|_2^2 < \alpha_k\|\tilde{x}_k-x^*\|^2_2,
\end{equation*}
where $x^*$ is the optimal solution to \eqref{eq:problem}.
\end{lemma}

\begin{proof}
Let $v_{i_t}^t=\nabla f_{i_t}(x_t) - \nabla f_{i_t}(\tilde{x}_k) + \nabla F(\tilde{x}_k)$ for the $k$-th epoch of SVRG-I or SVRG-BB. Then,
\begin{align*}
    \E\|v_{i_t}^t\|^2_2 = &\E\left\|(\nabla f_{i_t}(x_t) - \nabla f_{i_t}(x^*)) \right.-\left.(\nabla f_{i_t}(\tilde{x}_k) - \nabla f_{i_t}(x^*)) + \nabla F(\tilde{x}_k)\right\|^2_2 \\
    \leq& 2\E\left\|\nabla f_{i_t}(x_t) - \nabla f_{i_t}(x^*)\right\|^2_2 +  4\E\left\|\nabla f_{i_t}(\tilde{x}_k) - \nabla f_{i_t}(x^*)\right\|^2_2 +
    4\|\nabla F(\tilde{x}_k)\|^2_2 \\
    \leq& 2L\E \left[(x_t-x^*)^\T(\nabla f_i(x_t)-\nabla f_i(x^*) )\right] +4L^2\|\tilde{x}_k-x^*\|^2_2+4L^2\|\tilde{x}_k-x^*\|^2_2 \\
    =& 2L (x_t-x^*)^\T \nabla F(x_t) + 8L^2\|\tilde{x}_k-x^*\|^2_2,
\end{align*}
where in the first inequality we used the inequality $(a-b)^2\leq 2a^2+2b^2$ twice, in the second inequality we applied Lemma \ref{co-coercivity} to $f_{i_t}(x)$ and used the Lipschitz continuity of $\nabla f_{i_t}$ and $\nabla F$, and in the last equality we used the facts that $\E[\nabla f_{i_t}(x)]=\nabla F(x)$ and $\nabla F(x^*)=0$.

In the next, we bound the distance of $x_{t+1}$ to $x^*$ conditioned on $x_t$ and $\tilde{x}_k$.
\begin{align*}
    &\E\|x_{t+1} - x^*\|^2_2 \\
    =&\E\|x_t-\eta_k v_{i_t}^t-x^*\|^2_2 \\
    =& \|x_t-x^*\|^2_2-2\eta_k\E[(x_t-x^*)^\T v_{i_t}^t] + \eta_k^2\E\|v_{i_t}^t\|^2_2 \\
    =& \|x_t-x^*\|^2_2-2\eta_k (x_t-x^*)^\T \nabla F(x_t) + \eta_k^2\E\|v_{i_t}^t\|^2_2 \\
    \leq& \|x_t-x^*\|^2_2-2\eta_k (x_t-x^*)^\T \nabla F(x_t)
     + 2\eta_k^2L (x_t-x^*)^\T\nabla F(x_t) + 8\eta_k^2L^2\|\tilde{x}_k-x^*\|^2_2 \\
    =&\|x_t-x^*\|^2_2 - 2\eta_k(1-\eta_k L) (x_t-x^*)^\T \nabla F(x_t)   + 8\eta_k^2L^2\|\tilde{x}_k-x^*\|^2_2 \\
    \leq& \|x_t-x^*\|^2_2 - 2\eta_k\mu(1-\eta L)\|x_t-x^*\|^2+  8\eta_k^2L^2\|\tilde{x}_k-x^*\|^2_2 \\
    =& [1-2\eta_k\mu(1-\eta_k L)]\|x_t-x^*\|^2_2 + 8\eta_k^2L^2\|\tilde{x}_k-x^*\|^2_2,
\end{align*}
where in the third equality we used the fact that $\E[v_{i_t}^t]=\nabla F(x_t)$, and in the second inequality we used the strong convexity of $F(x)$.

By recursively applying the above inequality over $t$, and noting that $\tilde{x}_k=x_0$ and $\tilde{x}_{k+1}=x_m$, we can obtain
\begin{align*}
    &\E\|\tilde{x}_{k+1}-x^*\|^2_2 \\
    \leq&
    \left[1-2\eta_k\mu(1-\eta L)\right]^m\|\tilde{x}_k-x^*\|^2_2 +
    8\eta_k^2L^2\sum_{j=0}^{m-1}\left[1-2\eta_k\mu(1-\eta L)\right]^j\|\tilde{x}_k-x^*\|^2_2 \\
    <& \left[\left(1-2\eta_k\mu(1-\eta L)\right)^m + \frac{4\eta_k L^2}{\mu(1-\eta_k L)}\right]\|\tilde{x}_k-x^*\|^2_2 \\
    =&\alpha_k\|\tilde{x}_k-x^*\|^2_2. 
\end{align*}
\end{proof}

The linear convergence of SVRG-I follows immediately.
\begin{corollary}\label{cor-linear-conv-svrg-i}
In SVRG-I, if $m$ and $\eta$ are chosen such that
\begin{equation}
    \alpha:=\left(1-2\eta\mu(1-\eta L)\right)^m + \frac{4\eta L^2}{\mu(1-\eta L)} < 1,
\end{equation}
then SVRG-I (Algorithm \ref{alg:svrg} with Option I) converges linearly in expectation:
\begin{equation*}
    \E\left\|\tilde{x}_{k} - x^*\right\|_2^2 < \alpha^k\|\tilde{x}_0-x^*\|^2_2.
\end{equation*}
\end{corollary}

\begin{remark}
We now give some remarks on this convergence result.
\begin{enumerate}
\item To the best of our knowledge, this is the first time that the linear convergence of SVRG-I is established.
\item The condition required in \eqref{eq:def-alpha} is different from the condition required in \eqref{convergence-svrg-ii-alpha} for SVRG-II. As $m\rightarrow +\infty$, the first term in \eqref{convergence-svrg-ii-alpha} converges to 0 sublinearly, while the first term in \eqref{eq:def-alpha} converges to 0 linearly. On the other hand, the second term in \eqref{convergence-svrg-ii-alpha} reveals that $m$ depends on the condition number $L/\mu$ linearly, while the second term in \eqref{eq:def-alpha} suggests that $m$ depends on condition number $L/\mu$ quadratically. As a result, if the problem is ill-conditioned, then the convergence rate given in Corollary \ref{cor-linear-conv-svrg-i} might be slow.
\item The convergence result given in Corollary \ref{cor-linear-conv-svrg-i} is for the iterates $\tilde{x}_k$, while the one given in Theorem \ref{thm:svrg} is for the objective function values $F(\tilde{x}_k)$.
\end{enumerate}
\end{remark}

The following theorem establishes the linear convergence of SVRG-BB (Algorithm \ref{alg:svrg-bb}). 

\begin{theorem}\label{thm:svrg-bb}
Denote $\theta=(1-e^{-2\mu/L})/2$. It is easy to see that $\theta\in(0,1/2)$. In SVRG-BB, if $m$ is chosen such that
\begin{equation}
m>\max\left\{\frac{2}{\log(1-2\theta)+2\mu/L},\ \frac{4L^2}{\theta\mu^2}+\frac{L}{\mu}\right\},
\label{eq:cond-m}
\end{equation}
then SVRG-BB (Algorithm \ref{alg:svrg-bb}) converges linearly in expectation:
\begin{equation*}
    \E\left\|\tilde{x}_{k} - x^*\right\|_2^2 < (1-\theta)^k\|\tilde{x}_0-x^*\|^2_2.
\end{equation*}
\end{theorem}

\begin{proof}
Using the strong convexity of function $F(x)$, it is easy to obtain the following upper bound for the BB step size computed in Algorithm \ref{alg:svrg-bb}.
\begin{align*}
    \eta_k&=\frac{1}{m}\cdot\frac{\|\tilde{x}_k-\tilde{x}_{k-1}\|^2_2}{(\tilde{x}_k-\tilde{x}_{k-1})^\T(g_k-g_{k-1})} \\
    &\leq \frac{1}{m}\cdot\frac{\|\tilde{x}_k-\tilde{x}_{k-1}\|^2_2}{\mu\|\tilde{x}_k-\tilde{x}_{k-1}\|_2^2} = \frac{1}{m\mu}.
\end{align*}
Similarly, by the $L$-Lipschitz continuity of $\nabla F(x)$, it is easy to obtain that $\eta_k$ is uniformly lower bounded by $1/(mL)$. Therefore, $\alpha_k$ in \eqref{eq:def-alpha} can be bounded as:
\begin{align*}
\alpha_k\leq& \left[1-\frac{2\mu}{mL}\left(1-\frac{L}{m\mu}\right)\right]^m + \frac{4L^2}{m\mu^2[1-L/(m\mu)]} \\
\leq& \exp\left\{-\frac{2\mu}{mL}\left(1-\frac{L}{m\mu}\right)\cdot m\right\} + \frac{4L^2}{m\mu^2[1-L/(m\mu)]} \\
=& \exp\left\{-\frac{2\mu}{L}+\frac{2}{m}\right\}+\frac{4L^2}{m\mu^2-L\mu},
\end{align*}
Substituting \eqref{eq:cond-m} into the above inequality yields
\begin{align*}
\alpha_k < \exp\left\{-\frac{2\mu}{L}+\log(1-2\theta)+\frac{2\mu}{L}\right\} +\frac{4L^2}{4L^2/\theta+L\mu-L\mu} = (1-2\theta)+\theta =1-\theta.
\end{align*}
The desired result follows by applying Lemma \ref{thm:new-svrg}.
\end{proof}

\section{Barzilai-Borwein Step Size for SGD}\label{sec:sgd-bb}

In this section, we propose to incorporate the BB method to SGD \eqref{sgd}. The BB method does not apply to SGD directly, because SGD never computes the full gradient $\nabla F(x)$. In SGD, $\nabla f_{i_t}(x_t)$ is an unbiased estimation for $\nabla F(x_t)$ when $i_t$ is uniformly sampled (see \cite{Needell-NIPS-2014,Zhang-importance-sampling-2015} for studies on importance sampling, which does not sample $i_t$ uniformly). Therefore, one may suggest to use $\nabla f_{i_{t+1}}(x_{t+1})-\nabla f_{i_t}(x_t)$ to estimate $\nabla F(x_{t+1})-\nabla F(x_t)$ when computing the BB step size using formula \eqref{bb1-eta}. However, this approach does not work well because of the variance of the stochastic gradient estimates. 
The recent work by Sopy{\l}a and Drozda \cite{sopyla2015stochastic} suggested several variants of this idea to compute an estimated BB step size using the stochastic gradients. However, these ideas lack theoretical justifications and the numerical results in \cite{sopyla2015stochastic} show that these approaches are inferior to existing methods such as averaged SGD \cite{Polyak-average-sgd-1992}. 

The SGD-BB algorithm we propose in this paper works in the following manner. We call every $m$ iterations of SGD as one epoch. Following the idea of SVRG-BB, SGD-BB also uses the same step size computed by the BB formula in every epoch.
Our SGD-BB algorithm is described as in Algorithm \ref{alg:sgd-bb}.

\begin{algorithm}[ht]
    \caption{SGD with BB step size (SGD-BB)}\label{alg:sgd-bb}
    \begin{algorithmic}
        \STATE \textbf{Parameters}: update frequency $m$, initial step sizes $\eta_0$\ and $\eta_1$ (only used in the first two epochs), weighting parameter $\beta\in(0,1)$, initial point $\tilde{x}_0$
        \FOR{$k=0, 1,\cdots$}
            \IF{$k>0$}
                \STATE $\eta_k=\frac{1}{m}\cdot\|\tilde{x}_k-\tilde{x}_{k-1}\|^2_2/|(\tilde{x}_k-\tilde{x}_{k-1})^\T(\hat{g}_k-\hat{g}_{k-1})|$
            \ENDIF
            \STATE $x_0 = \tilde{x}_k$
            \STATE $\hat{g}_{k+1}=0$
            \FOR{$t=0,\cdots,m-1$}
                \STATE Randomly pick $i_t\in \{1,\dots,n\}$
                \STATE $x_{t+1}=x_t- {\eta}_k \nabla f_{i_t}(x_t)$ \hfill $(*)$
                \STATE $\hat{g}_{k+1}=\beta \nabla f_{i_t}(x_t) + (1-\beta) \hat{g}_{k+1}$
            \ENDFOR
            \STATE $\tilde{x}_{k+1}=x_{m}$
        \ENDFOR
   \end{algorithmic}
\end{algorithm}

\begin{remark}
We have a few remarks about SGD-BB (Algorithm \ref{alg:sgd-bb}).
\begin{enumerate}
\item SGD-BB takes the average of the stochastic gradients in one epoch as an estimation of the full gradient.
\item Note that for computing $\eta_k$ in Algorithm \ref{alg:sgd-bb}, we actually take the absolute value for the BB formula \eqref{bb1-eta}. This is because that unlike SVRG-BB, $\hat{g}_k$ in Algorithm \ref{alg:sgd-bb} is the average of $m$ stochastic gradients at different iterates, not an exact full gradient. As a result, the step size generated by \eqref{bb1-eta} can be negative. This can be seen from the following argument. Suppose $\beta$ is chosen such that
    \begin{equation} \label{eq:average}
    \hat{g}_k=\frac{1}{m}\sum_{t=0}^{m-1} \nabla f_{i_t}(x_t),
    \end{equation}
    where we use the same notation as in Algorithm \ref{alg:svrg-bb} and $x_t$ $(t=0,1,\ldots,m-1)$ denote the iterates in the $(k-1)$-st epoch.
     From \eqref{eq:average}, it is easy to see that
\begin{equation*}
    \tilde{x}_k-\tilde{x}_{k-1} = -m\eta_{k-1}\hat{g}_k.
\end{equation*}
By substituting this equality into the equation for computing $\eta_k$ in Algorithm \ref{alg:sgd-bb}, we have
\begin{align*}
    \eta_k=&\frac{1}{m}\cdot\frac{\|\tilde{x}_k-\tilde{x}_{k-1}\|^2}{|(\tilde{x}_k-\tilde{x}_{k-1})^\T (\hat{g}_k-\hat{g}_{k-1})|} \\
    =&\frac{1}{m}\cdot\frac{\|-m\eta_{k-1}\hat{g}_k\|^2}{|(-m\eta_{k-1}\hat{g}_k)^\T (\hat{g}_k-\hat{g}_{k-1} )|} \\
    =&\frac{\eta_{k-1}}{\left| 1- \hat{g}_k^\T \hat{g}_{k-1}/\|\hat{g}_k\|_2^2\right|}. \numberthis \label{eq:sgd-bb}
\end{align*}
Without taking the absolute value, the denominator of \eqref{eq:sgd-bb} is $\hat{g}_k^\T \hat{g}_{k-1}/\|\hat{g}_k\|_2^2-1$, which can be negative in stochastic settings. 
\item Moreover, from \eqref{eq:sgd-bb} we have the following observations.
If $\hat{g}_k^\T\hat{g}_{k-1} < 0$, then $\eta_k$ is smaller than $\eta_{k-1}$. This is reasonable because $\hat{g}_k^\T\hat{g}_{k-1} < 0$ indicates that the step size is too large and we need to shrink it. If $\hat{g}_k^\T\hat{g}_{k-1} > 0$, then it indicates that we should be more aggressive to take larger step size. We found from our numerical experiments that when the iterates are close to optimum, the size of $\hat{g}_k$ and $\hat{g}_{k-1}$ do not differentiate much. As a result, $\eta_k$ is usually increased from $\eta_{k-1}$ by using \eqref{eq:sgd-bb}. Hence, the way we compute $\eta_k$ in Algorithm \ref{alg:sgd-bb} is in a sense to dynamically adjust the step size, by evaluating whether we are moving the iterates along the right direction. This kind of idea can be traced back to \cite{Kesten-1958}.
\item Furthermore, in order to make sure the averaged stochastic gradients $\hat{g}_k$ in \eqref{eq:average} is close to $\nabla F(\tilde{x}_k)$, it is natural to emphasize more on the latest sample gradients. Therefore, in Algorithm \ref{alg:sgd-bb} we update $\hat{g}_k$ recursively using
\begin{equation*}
    \hat{g}_{k+1}=\beta \nabla f_{i_t}(x_t) + (1-\beta) \hat{g}_{k+1},
\end{equation*}
starting from $\hat{g}_{k+1} =0$, where $\beta\in(0,1)$ is a weighting parameter.
\end{enumerate}
\end{remark}

Note that SGD-BB requires the averaged gradients in two epochs to compute the BB step size, which can only be done starting from the third epoch. Therefore, we need to specify the step sizes $\eta_0$ and $\eta_1$ for the first two epochs. From our numerical experiments, we found that the performance of SGD-BB is not sensitive to choices of $\eta_0$ and $\eta_1$. 

\subsection{Smoothing Technique for the Step Sizes} \label{sub:smoothing}
Due to the randomness of the stochastic gradients, the step size computed in SGD-BB may vibrate drastically sometimes and this may cause instability of the algorithm. Inspired by \cite{masse2015speed}, we propose the following smoothing technique to stabilize the step size.

It is known that in order to guarantee the convergence of SGD, the step sizes are required to be diminishing. Similar as in \cite{masse2015speed}, we assume the step sizes are in the form of $C/\phi(k)$, where $C>0$ is an unknown constant that needs to be estimated, $\phi(k)$ is a pre-specified function that controls the decreasing rate of the step size, and a typical choice of function $\phi$ is $\phi(k)=k+1$.
In the $k$-th epoch of Algorithm \ref{alg:sgd-bb}, we have all the previous step sizes $\eta_2,\eta_3, \dots, \eta_k$ generated by the BB method, while the step sizes generated by the function $C/\phi(k)$ are given by $C/\phi(2), C/\phi(3),\ldots,C/\phi(k)$. In order to ensure that these two sets of step sizes are close to each other, we solve the following optimization problem to determine the unknown parameter $C$:
\be\label{smooth-log}
    \hat{C}_k:=\argmin_{C}\ \sum_{j=2}^k\left[\log{\frac{C}{\phi(j)}}-\log{\eta_j}\right]^2.
\ee
Here we take the logarithms of the step sizes to ensure that the estimation is not dominated by those $\eta_j$'s with large magnitudes. It is easy to verify that the solution to \eqref{smooth-log} is given by
\begin{equation*}
    \hat{C}_k=\prod_{j=2}^k \left[\eta_j\phi(j) \right]^{1/(k-1)}.
\end{equation*}
Therefore, the smoothed step size for the $k$-th epoch of Algorithm \ref{alg:sgd-bb} is:
\begin{equation}\label{smooth-eta}
    \tilde{\eta}_k=\hat{C}_k/\phi(k)=\prod_{j=2}^k \left[\eta_j\phi(j) \right]^{1/(k-1)}/\phi(k).
\end{equation}
That is, we replace the $\eta_k$ in equation $(*)$ of Algorithm \ref{alg:sgd-bb} by $\tilde{\eta}_k$ in \eqref{smooth-eta}.

In practice, we do not need to store all the $\eta_j$'s and $\hat{C}_k$ can be computed recursively by
\begin{equation*}
    \hat{C}_k=\hat{C}_{k-1}^{(k-2)/(k-1)}\cdot \left[\eta_k\phi(k)\right]^{1/(k-1)}.
\end{equation*}


\subsection{Incorporating BB Step Size to SGD Variants}

The BB step size and the smoothing technique we used in SGD-BB (Algorithm \ref{alg:sgd-bb}) can also be used in other variants of SGD. In this section, we use SAG as an example to illustrate how to incorporate the BB step size.
SAG with BB step size (denoted as SAG-BB) is described as in Algorithm \ref{alg:sag-bb}. 
Because SAG does not need diminishing step sizes to ensure convergence, in the smoothing technique we just choose $\phi(k)\equiv 1$. In this case, the smoothed step size $\tilde{\eta}_k$ is equal to the geometric mean of all previous BB step sizes.

\begin{algorithm}[ht]
    \caption{SAG with BB step size (SAG-BB)}\label{alg:sag-bb}
    \begin{algorithmic}
        \STATE \textbf{Parameters}: update frequency $m$, initial step sizes $\eta_0$\ and $\eta_1$ (only used in the first two epochs), weighting parameter $\beta\in(0,1)$, initial point $\tilde{x}_0$
        \STATE $y_i=0$ for $i=1,\dots,n$
        \FOR{$k=0, 1,\cdots$}
            \IF{$k>0$}
                \STATE $\eta_k=\frac{1}{m}\cdot\|\tilde{x}_k-\tilde{x}_{k-1}\|^2_2/|(\tilde{x}_k-\tilde{x}_{k-1})^\T(\hat{g}_k-\hat{g}_{k-1})|$
                \STATE $\tilde{\eta}_k=\left(\prod_{j=2}^k \eta_j \right)^{\frac{1}{k-1}}$ \hfill $\triangleright$ smoothing technique
            \ENDIF
            \STATE $x_0 = \tilde{x}_k$
            \STATE $\hat{g}_{k+1}=0$
            \FOR{$t=0,\cdots,m-1$}
                \STATE Randomly pick $i_t\in \{1,\dots,n\}$
                \STATE $y_{i_t}=\nabla f_{i_t}(x_t)$
                \STATE $x_{t+1}=x_t- \frac{{\eta}_k}{n}\sum_{i=1}^n y_i$ \hfill $\triangleright$ SAG update
                \STATE $\hat{g}_{k+1}=\beta \nabla f_{i_t}(x_t) + (1-\beta) \hat{g}_{k+1}$
            \ENDFOR
            \STATE $\tilde{x}_{k+1}=x_{m}$
        \ENDFOR
   \end{algorithmic}
\end{algorithm}

\section{Numerical Experiments}\label{sec:num}

In this section, we conduct some numerical experiments to demonstrate the efficacy of our SVRG-BB (Algorithm \ref{alg:svrg-bb}) and SGD-BB (Algorithm \ref{alg:sgd-bb}) algorithms. In particular, we apply SVRG-BB and SGD-BB to solve two standard testing problems in machine learning: logistic regression with $\ell_2$-norm regularization
\be\label{log-reg}
    (\text{\sl LR}) \qquad \min_x \ F(x)=\frac{1}{n}\sum_{i=1}^n \log{\left[1+\mathrm{exp}(-b_i a_i^\T x)\right]}+\frac{\lambda}{2}\|x\|^2_2,
\ee
and the squared hinge loss SVM with $\ell_2$-norm regularization
\be\label{svm}
    (\text{\sl SVM}) \qquad \min_x \ F(x)=\frac{1}{n}\sum_{i=1}^n \left([1-b_i a_i^\T x]_+\right)^2+\frac{\lambda}{2}\|x\|^2_2,
\ee
where $a_i\in \mathbb{R}^d$ and $b_i\in \{\pm 1\}$ are the feature vector and class label of the $i$-th sample, respectively, and $\lambda>0$ is a weighting parameter.

We tested SVRG-BB and SGD-BB for \eqref{log-reg} and \eqref{svm} for three standard real data sets, which were downloaded from the LIBSVM website\footnote{\url{www.csie.ntu.edu.tw/~cjlin/libsvmtools/}.}.
Detailed information of these three data sets are given in Table \ref{tab:data}. 

\begin{table}[ht]
\center
\caption{Data and model information of the experiments}
\label{tab:data}
\begin{tabular}{ |c|c|c|c|c| }
 \hline
 Dataset & $n$ & $d$ & model & $\lambda$  \\ \hline
 rcv1.binary & 20,242 & 47,236 & \textsl{LR} & $10^{-5}$ \\
 w8a & 49,749 & 300 & \textsl{LR}  & $10^{-4}$ \\
 ijcnn1 & 49,990 & 22 & \textsl{SVM}  & $10^{-4}$ \\
 \hline
\end{tabular}
\end{table}


\subsection{Numerical Results of SVRG-BB}

\begin{figure}[t]
\begin{subfigure}[h]{0.32\textwidth}
\includegraphics[width=0.95\linewidth]{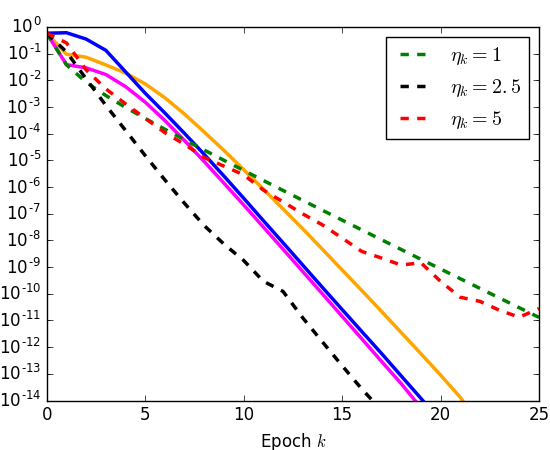}
\caption{Sub-optimality on rcv1.binary}
\label{fig:svrg_rcv1_1}
\end{subfigure}
\begin{subfigure}[h]{0.32\textwidth}
\includegraphics[width=0.95\linewidth]{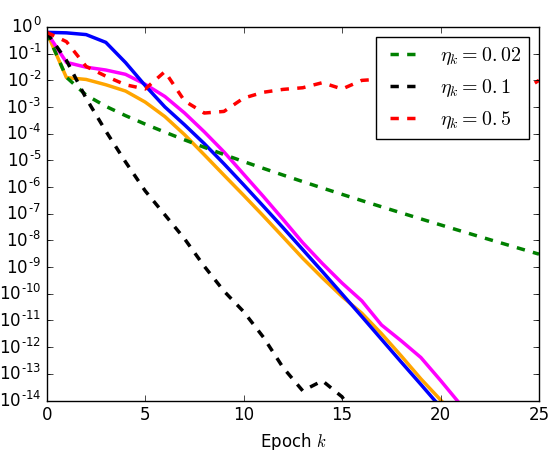}
\caption{Sub-optimality on w8a}
\label{fig:svrg_w8a_1}
\end{subfigure}
\begin{subfigure}[h]{0.32\textwidth}
\includegraphics[width=0.95\linewidth]{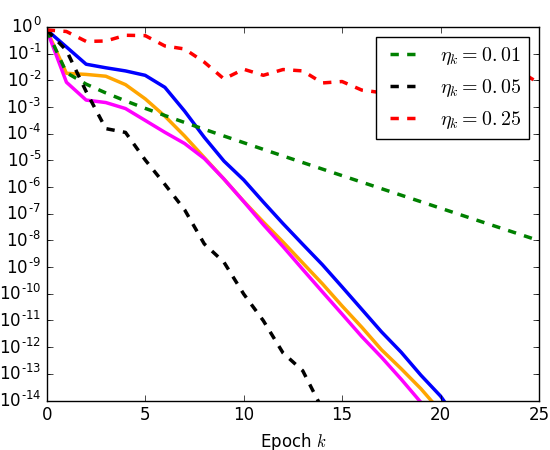}
\caption{Sub-optimality on ijcnn1}
\label{fig:svrg_ijcnn1_1}
\end{subfigure}

\begin{subfigure}[h]{0.32\textwidth}
\includegraphics[width=0.95\linewidth]{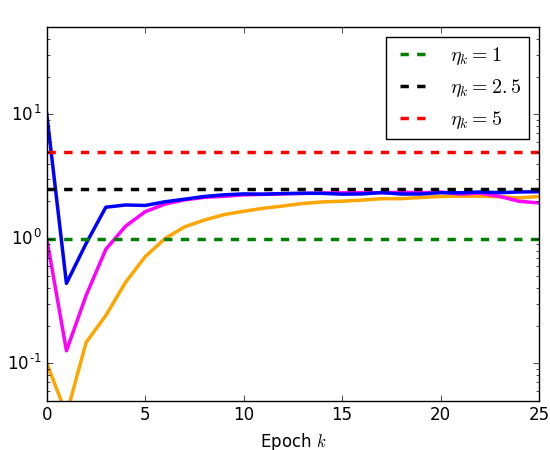}
\caption{Step sizes on rcv1.binary}
\label{fig:svrg_rcv1_2}
\end{subfigure}
\begin{subfigure}[h]{0.32\textwidth}
\includegraphics[width=0.95\linewidth]{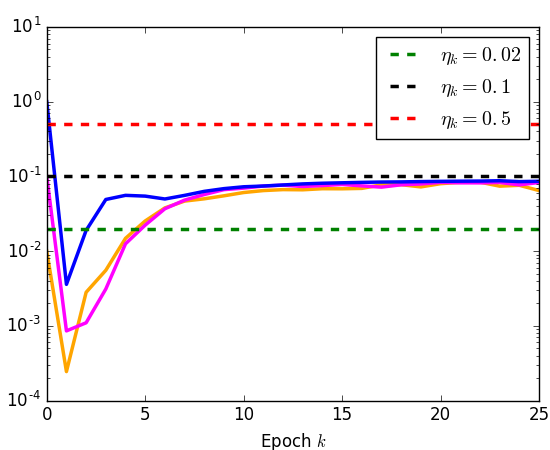}
\caption{Step sizes on w8a}
\label{fig:svrg_w8a_2}
\end{subfigure}
\begin{subfigure}[h]{0.32\textwidth}
\includegraphics[width=0.95\linewidth]{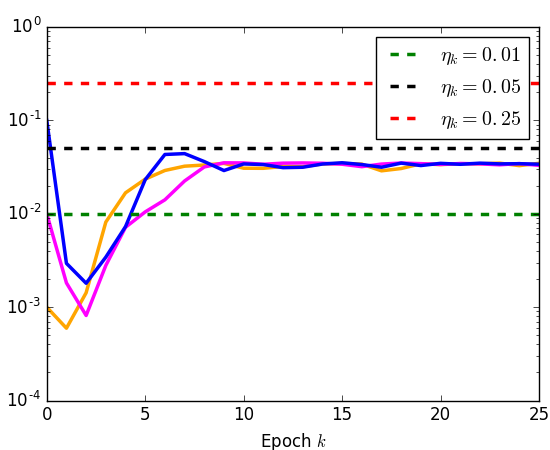}
\caption{Step sizes on ijcnn1}
\label{fig:svrg_ijcnn1_2}
\end{subfigure}

\caption{Comparison of SVRG-BB and SVRG with fixed step sizes on different problems. The dashed lines stand for SVRG with different fixed step sizes $\eta_k$ given in the legend. The solid lines stand for SVRG-BB with different $\eta_0$; for example, the solid lines in Sub-figures \subref{fig:svrg_rcv1_1} and \subref{fig:svrg_rcv1_2} correspond to SVRG-BB with $\eta_0=10,1,0.1$, respectively.}
\label{fig:svrg}
\end{figure}

In this section, we compare SVRG-BB (Algorithm \ref{alg:svrg-bb}) with SVRG (Algorithm \ref{alg:svrg}) for solving \eqref{log-reg} and \eqref{svm}. We used the best-tuned step size for SVRG, and three different initial step sizes $\eta_0$ for SVRG-BB. For both SVRG-BB and SVRG, we set $m=2n$ as suggested in \cite{svrg}.

The comparison results of SVRG-BB and SVRG are shown in Figure \ref{fig:svrg}. In all the six sub-figures, the $x$-axis denotes the number of epochs $k$, i.e., the number of outer loops in Algorithm \ref{alg:svrg-bb}. In Figures \ref{fig:svrg_rcv1_1}, \ref{fig:svrg_w8a_1} and \ref{fig:svrg_ijcnn1_1}, the $y$-axis denotes the sub-optimality $F(\tilde{x}_k)-F(x^*)$, and in Figures \ref{fig:svrg_rcv1_2}, \ref{fig:svrg_w8a_2} and \ref{fig:svrg_ijcnn1_2}, the $y$-axis denotes the step size $\eta_k$.
Note that $x^*$ is obtained by running SVRG with the best-tuned step size until it converges, which is a common practice in the testing of stochastic gradient descent methods. In all the six sub-figures, the dashed lines correspond to SVRG with fixed step sizes given in the legends of the figures. Moreover, the dashed lines in black color always represent SVRG with best-tuned fixed step size, and the green dashed lines use a smaller fixed step size, and the red dashed lines use a larger fixed step size, compared with the best-tuned ones. The solid lines correspond to SVRG-BB with different initial step sizes $\eta_0$. The solid lines with blue, purple and yellow colors in Figures \ref{fig:svrg_rcv1_1} and \ref{fig:svrg_rcv1_2} correspond to $\eta_0 = 10$, $1$, and $0.1$, respectively; the solid lines with blue, purple and yellow colors in Figures \ref{fig:svrg_w8a_1} and \ref{fig:svrg_w8a_2} correspond to $\eta_0 = 1$, $0.1$, and $0.01$, respectively; the solid lines with blue, purple and yellow colors in Figures \ref{fig:svrg_ijcnn1_1} and \ref{fig:svrg_ijcnn1_2} correspond to $\eta_0 = 0.1$, $0.01$, and $0.001$, respectively.

It can be seen from Figures \ref{fig:svrg_rcv1_1}, \ref{fig:svrg_w8a_1} and \ref{fig:svrg_ijcnn1_1} that, SVRG-BB can always achieve the same level of sub-optimality as SVRG with the  best-tuned step size. Although SVRG-BB needs slightly more epochs compared with SVRG with the best-tuned step size, it clearly outperforms SVRG with the other two choices of step sizes. Moreover, from Figures \ref{fig:svrg_rcv1_2}, \ref{fig:svrg_w8a_2} and \ref{fig:svrg_ijcnn1_2} we see that the step sizes computed by SVRG-BB converge to the best-tuned step sizes after about 10 to 15 epochs. From Figure \ref{fig:svrg} we also see that SVRG-BB is not sensitive to the choice of $\eta_0$.
Therefore, SVRG-BB has very promising potential in practice because it generates the best step sizes automatically while running the algorithm.

\subsection{Numerical Results of SGD-BB}

\begin{figure}[t]
\begin{subfigure}[h]{0.32\textwidth}
\includegraphics[width=0.95\linewidth]{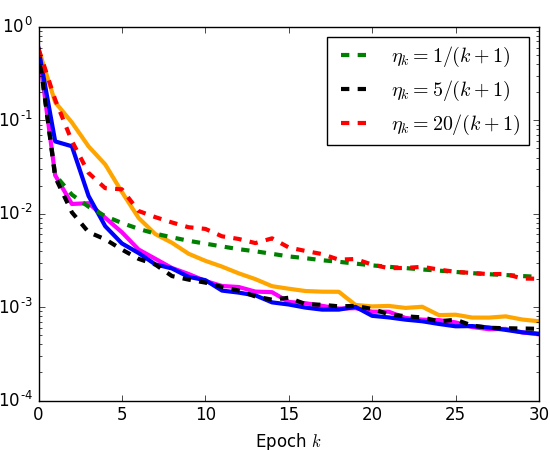}
\caption{Sub-optimality on rcv1.binary}
\label{fig:sgd_rcv1_1}
\end{subfigure}
\begin{subfigure}[h]{0.32\textwidth}
\includegraphics[width=0.95\linewidth]{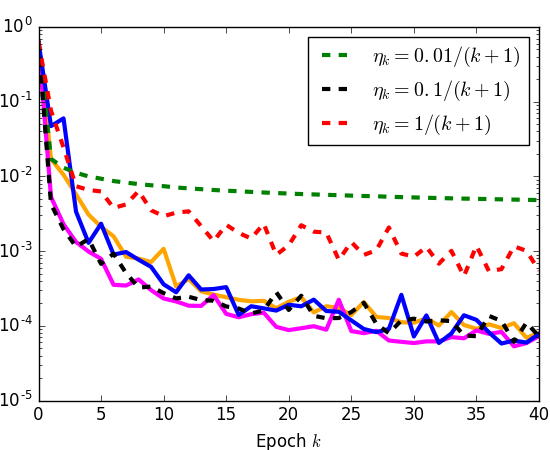}
\caption{Sub-optimality on w8a}
\label{fig:sgd_w8a_1}
\end{subfigure}
\begin{subfigure}[h]{0.32\textwidth}
\includegraphics[width=0.95\linewidth]{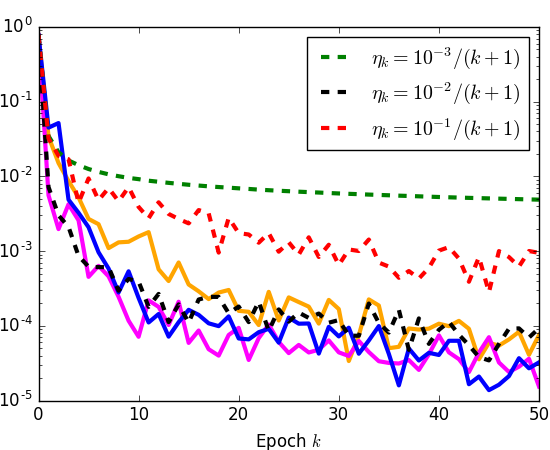}
\caption{Sub-optimality on ijcnn1}
\label{fig:sgd_ijcnn1_1}
\end{subfigure}

\begin{subfigure}[h]{0.32\textwidth}
\includegraphics[width=0.95\linewidth]{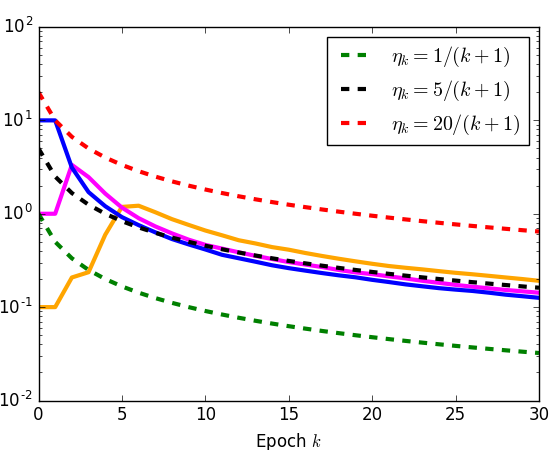}
\caption{Step sizes on rcv1.binary}
\label{fig:sgd_rcv1_2}
\end{subfigure}
\begin{subfigure}[h]{0.32\textwidth}
\includegraphics[width=0.95\linewidth]{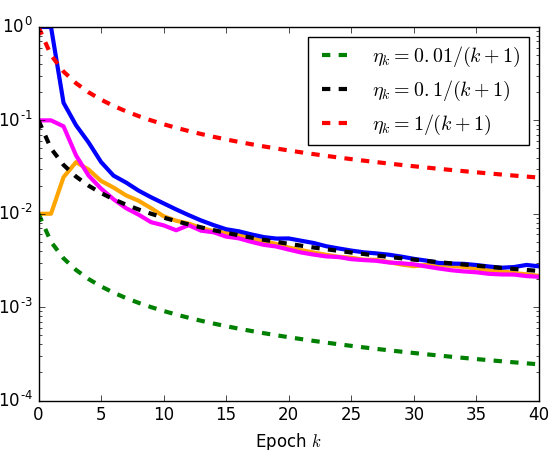}
\caption{Step sizes on w8a}
\label{fig:sgd_w8a_2}
\end{subfigure}
\begin{subfigure}[h]{0.32\textwidth}
\includegraphics[width=0.95\linewidth]{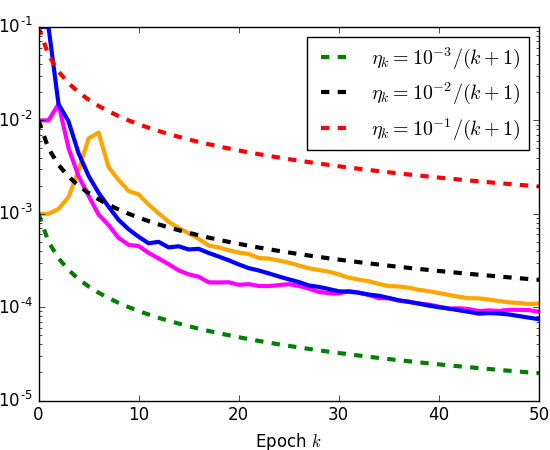}
\caption{Step sizes on ijcnn1}
\label{fig:sgd_ijcnn1_2}
\end{subfigure}

\caption{Comparison of SGD-BB and SGD. The dashed lines correspond to SGD with diminishing step sizes in the form $\eta/(k+1)$ with different constants $\eta$. The solid lines stand for SGD-BB with different initial step sizes $\eta_0$; for example, the solid lines in Sub-figure \subref{fig:sgd_rcv1_1} and \subref{fig:sgd_rcv1_2} correspond to SGD-BB with $\eta_0=10,1,0.1$, respectively.}
\label{fig:sgd}
\end{figure}

In this section, we compare SGD-BB with smoothing technique (Algorithm \ref{alg:sgd-bb}) with SGD for solving \eqref{log-reg} and \eqref{svm}. We set $m=n$, $\beta=10/m$ and $\eta_1=\eta_0$ in our experiments. We used $\phi(k)=k+1$ when applying the smoothing technique. Since SGD requires diminishing step size to converge, we tested SGD with diminishing step size in the form $\eta/(k+1)$ with different constants $\eta$. The comparison results are shown in Figure \ref{fig:sgd}. Similar as Figure \ref{fig:svrg}, the dashed line with black color represents SGD with the best-tuned $\eta$, and the green and red dashed lines correspond to the other two choices of $\eta$. The solid lines with blue, purple and yellow colors in Figures \ref{fig:sgd_rcv1_1} and \ref{fig:sgd_rcv1_2} correspond to $\eta_0 = 10$, $1$, and $0.1$, respectively; the solid lines with blue, purple and yellow colors in Figures \ref{fig:sgd_w8a_1} and \ref{fig:sgd_w8a_2} correspond to $\eta_0 = 1$, $0.1$, and $0.01$, respectively; the solid lines with blue, purple and yellow colors in Figures \ref{fig:sgd_ijcnn1_1} and \ref{fig:sgd_ijcnn1_2} correspond to $\eta_0 = 0.1$, $0.01$, and $0.001$, respectively.

From Figures \ref{fig:sgd_rcv1_1}, \ref{fig:sgd_w8a_1} and \ref{fig:sgd_ijcnn1_1} we can see that SGD-BB gives comparable or even better sub-optimality than SGD with best-tuned diminishing step size, and SGD-BB is significantly better than SGD with the other two choices of step size. From Figures \ref{fig:sgd_rcv1_2}, \ref{fig:sgd_w8a_2} and \ref{fig:sgd_ijcnn1_2} we see that after only a few epochs, the step sizes generated by SGD-BB approximately coincide with the best-tuned diminishing step sizes. It can also be seen that after only a few epochs, the step sizes are stabilized by the smoothing technique and they approximately follow the same decreasing trend as the best-tuned diminishing step sizes.

\subsection{Comparison with Other Methods}

\begin{figure}[t]
\begin{minipage}[h]{0.32\textwidth}
\centering
rcv1.binary
\end{minipage}
\begin{minipage}[h]{0.32\textwidth}
\centering
w8a
\end{minipage}
\begin{minipage}[h]{0.32\textwidth}
\centering
ijcnn1
\end{minipage}

\begin{subfigure}[h]{0.32\textwidth}
\includegraphics[width=0.95\linewidth]{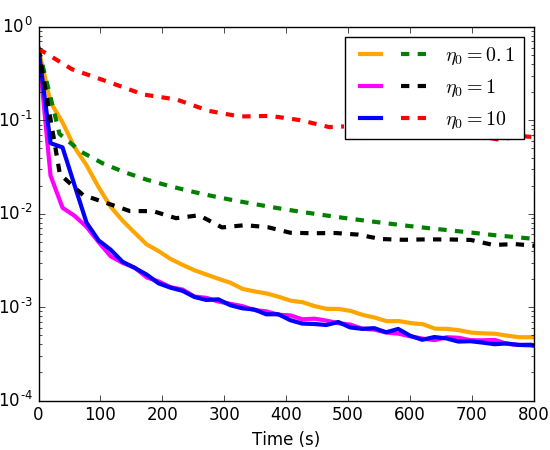}
\caption{AdaGrad versus SGD-BB}
\label{fig:AdaGrad_1}
\end{subfigure}
\begin{subfigure}[h]{0.32\textwidth}
\includegraphics[width=0.95\linewidth]{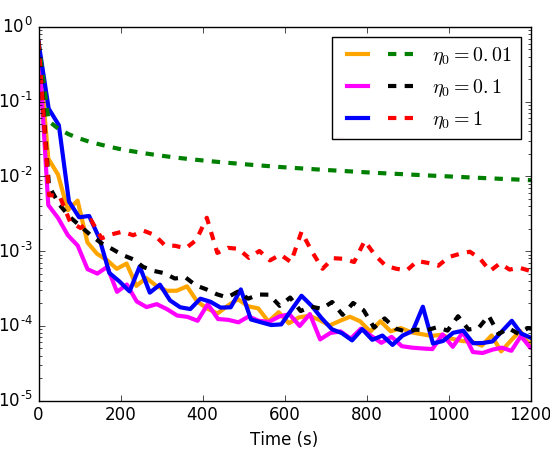}
\caption{AdaGrad versus SGD-BB}
\label{fig:AdaGrad_2}
\end{subfigure}
\begin{subfigure}[h]{0.32\textwidth}
\includegraphics[width=0.95\linewidth]{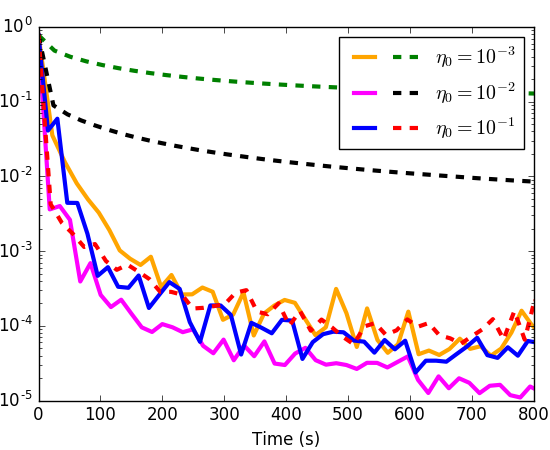}
\caption{AdaGrad versus SGD-BB}
\label{fig:AdaGrad_3}
\end{subfigure}

\begin{subfigure}[h]{0.32\textwidth}
\includegraphics[width=0.95\linewidth]{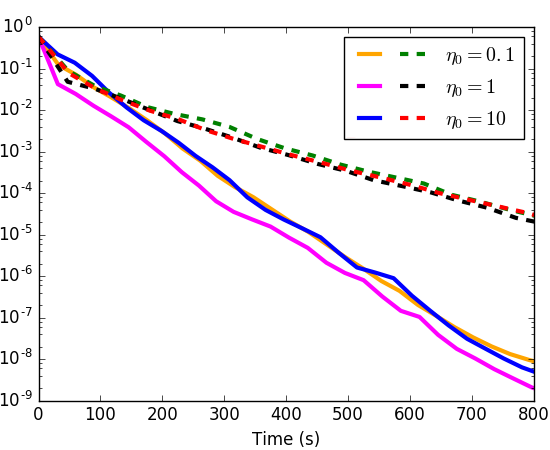}
\caption{SAG-L versus SAG-BB}
\label{fig:LineSearch_1}
\end{subfigure}
\begin{subfigure}[h]{0.32\textwidth}
\includegraphics[width=0.95\linewidth]{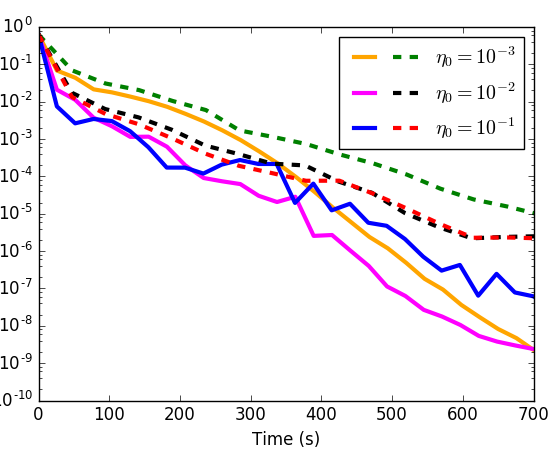}
\caption{SAG-L versus SAG-BB}
\label{fig:LineSearch_2}
\end{subfigure}
\begin{subfigure}[h]{0.32\textwidth}
\includegraphics[width=0.95\linewidth]{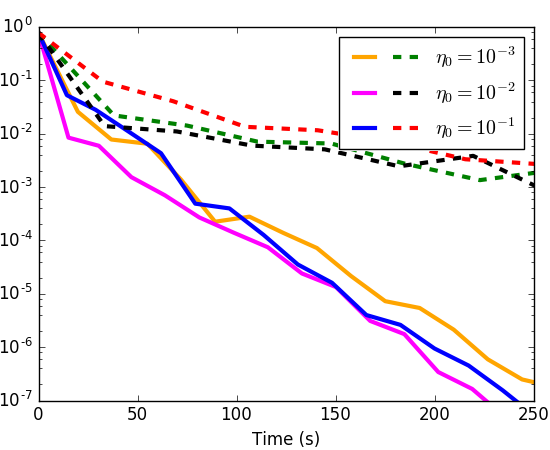}
\caption{SAG-L versus SAG-BB}
\label{fig:LineSearch_3}
\end{subfigure}

\begin{subfigure}[h]{0.32\textwidth}
\includegraphics[width=0.95\linewidth]{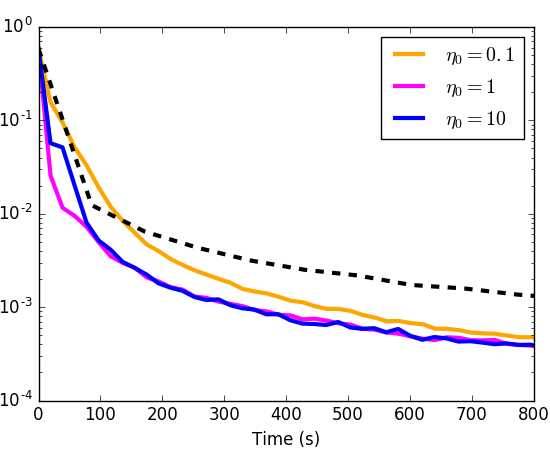}
\caption{oLBFGS versus SGD-BB}
\label{fig:oLBFGS_1}
\end{subfigure}
\begin{subfigure}[h]{0.32\textwidth}
\includegraphics[width=0.95\linewidth]{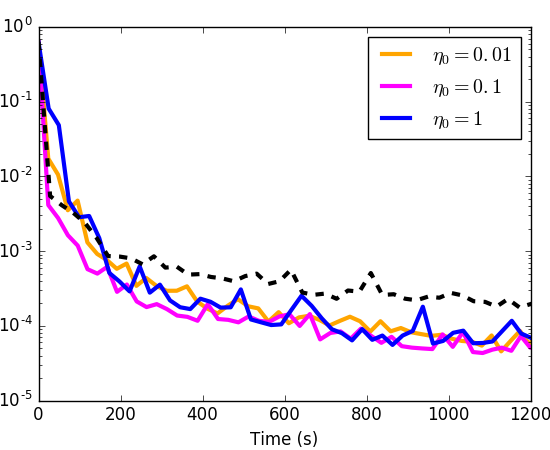}
\caption{oLBFGS versus SGD-BB}
\label{fig:oLBFGS_2}
\end{subfigure}
\begin{subfigure}[h]{0.32\textwidth}
\includegraphics[width=0.95\linewidth]{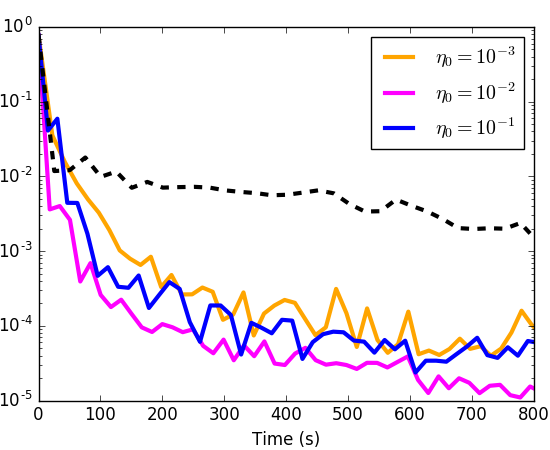}
\caption{oLBFGS versus SGD-BB}
\label{fig:oLBFGS_3}
\end{subfigure}

\caption{Comparison of SGD-BB and SAG-BB with three existing methods. The $x$-axes all denote the CPU time (in seconds). The $y$-axes all denote the sub-optimality $F(x_k)-F(x^*)$. In the first row, solid lines stand for SGD-BB, and dashed lines stand for AdaGrad; In the second row, solid lines stand for SAG-BB, and dashed lines stand for SAG with line search; In the third row, solid lines stand for SGD-BB, and the dashed lines stand for oLBFGS.}
\label{fig:comparison}
\end{figure}

In this section, we compare our SGD-BB (Algorithm \ref{alg:sgd-bb}) and SAG-BB (Algorithm \ref{alg:sag-bb}) with three existing methods: AdaGrad \cite{duchi2011adaptive}, SAG with line search (denoted as SAG-L) \cite{schmidt2013minimizing}, and a stochastic quasi-Newton method: oLBFGS \cite{schraudolph2007stochastic}. For both SGD-BB and SAG-BB, we set $m=n$ and $\beta=10/m$. Because these methods have very different per-iteration complexity, we compare their CPU time needed to achieve the same sub-optimality.  


Figures \ref{fig:AdaGrad_1}, \ref{fig:AdaGrad_2} and \ref{fig:AdaGrad_3} show the comparison results of SGD-BB and AdaGrad. From these figures we see that AdaGrad usually has a very quick start, but in many cases the convergence becomes slow in later iterations. Besides, AdaGrad is still somewhat sensitive to the initial step sizes. Especially, when a small initial step size is used, AdaGrad is not able to increase the step size to a suitable level. As a contrast, SGD-BB converges very fast in all three tested problems, and it is not sensitive to the initial step size $\eta_0$.

Figures \ref{fig:LineSearch_1}, \ref{fig:LineSearch_2} and \ref{fig:LineSearch_3} show the comparison results of SAG-BB and SAG-L. From these figures we see that the SAG-L is quite robust and is not sensitive to the choice of $\eta_0$. However, SAG-BB is much faster than SAG-L to reach the same sub-optimality on the tested problems. 

Figures \ref{fig:oLBFGS_1}, \ref{fig:oLBFGS_2} and \ref{fig:oLBFGS_3} show the comparison results of SGD-BB and oLBFGS. 
For oLBFGS we used a best-tuned step size. From these figures we see that oLBFGS is much slower than SGD-BB, which is mainly because oLBFGS needs more computational effort per iteration.

\section{Conclusion}\label{sec:con}
In this paper we proposed to use the BB method to compute the step sizes for SGD and SVRG, which leads to two new stochastic gradient methods: SGD-BB and SVRG-BB. We proved the linear convergence of SVRG-BB for strongly convex function, and as a by-product, we proved the linear convergence of the original SVRG with option I for strongly convex function. We also proposed a smoothing technique to stabilize the step sizes generated in SGD-BB, and we showed how to incorporate the BB method to other SGD variants such as SAG. We conducted numerical experiments on real data sets to compare the performance of SVRG-BB and SGD-BB with existing methods. The numerical results showed that the performance of our SVRG-BB and SGD-BB is comparable to and sometimes even better than the original SVRG and SGD with best-tuned step sizes, and is superior to some advanced SGD variants.

\bibliography{references}

\begin{thebibliography}{10}

\bibitem{allen2016variance}
Z.~Allen-Zhu and E.~Hazan.
\newblock Variance reduction for faster non-convex optimization.
\newblock {\em arXiv preprint arXiv:1603.05643}, 2016.

\bibitem{harikandeh2015stopwasting}
R.~Babanezhad, M.~O. Ahmed, A.~Virani, M.~Schmidt, K.~Kone{\v{c}}n{\`y}, and
  S.~Sallinen.
\newblock Stop wasting my gradients: Practical {SVRG}.
\newblock In {\em Advances in Neural Information Processing Systems}, pages
  2242--2250, 2015.

\bibitem{barzilai1988two}
J.~Barzilai and J.~M. Borwein.
\newblock Two-point step size gradient methods.
\newblock {\em IMA Journal of Numerical Analysis}, 8(1):141--148, 1988.

\bibitem{Dai-2013}
Y.-H. Dai.
\newblock A new analysis on the {B}arzilai-{B}orwein gradient method.
\newblock {\em Journal of Operations Research Society of China}, 1(2):187--198,
  2013.

\bibitem{Dai-Fletcher-2005}
Y.-H. Dai and R.~Fletcher.
\newblock Projected {B}arzilai-{B}orwein methods for large-scale
  box-constrained quadratic programming.
\newblock {\em Numerische Mathematik}, 100(1):21--47, 2005.

\bibitem{Dai-CBB-2006}
Y.-H. Dai, W.~W. Hager, K.~Schittkowski, and H.~Zhang.
\newblock The cyclic {B}arzilai-{B}orwein method for unconstrained
  optimization.
\newblock {\em IMA Journal of Numerical Analysis}, 26(3):604--627, 2006.

\bibitem{defazio2014saga}
A.~Defazio, F.~Bach, and S.~Lacoste-Julien.
\newblock {SAGA}: A fast incremental gradient method with support for
  non-strongly convex composite objectives.
\newblock In {\em Advances in Neural Information Processing Systems}, pages
  1646--1654, 2014.

\bibitem{duchi2011adaptive}
J.~Duchi, E.~Hazan, and Y.~Singer.
\newblock Adaptive subgradient methods for online learning and stochastic
  optimization.
\newblock {\em The Journal of Machine Learning Research}, 12:2121--2159, 2011.

\bibitem{fletcher2005barzilai}
R.~Fletcher.
\newblock On the {B}arzilai-{B}orwein method.
\newblock In {\em Optimization and control with applications}, pages 235--256.
  Springer, 2005.

\bibitem{svrg}
R.~Johnson and T.~Zhang.
\newblock Accelerating stochastic gradient descent using predictive variance
  reduction.
\newblock In {\em Advances in Neural Information Processing Systems}, pages
  315--323, 2013.

\bibitem{Kesten-1958}
H.~Kesten.
\newblock Accelerated stochastic approximation.
\newblock {\em The Annals of Mathematical Statistics}, 29(1):41--59, 1958.

\bibitem{konevcny2013semi}
J.~Kone{\v{c}}n{\`y} and P.~Richt{\'a}rik.
\newblock Semi-stochastic gradient descent methods.
\newblock {\em arXiv preprint arXiv:1312.1666}, 2013.

\bibitem{mahsereci2015probabilistic}
M.~Mahsereci and P.~Hennig.
\newblock Probabilistic line searches for stochastic optimization.
\newblock {\em arXiv preprint arXiv:1502.02846}, 2015.

\bibitem{masse2015speed}
P.~Y. Mass{\'e} and Y.~Ollivier.
\newblock Speed learning on the fly.
\newblock {\em arXiv preprint arXiv:1511.02540}, 2015.

\bibitem{Needell-NIPS-2014}
D.~Needell, N.~Srebro, and R.~Ward.
\newblock Stochastic gradient descent, weighted sampling, and the randomized
  kaczmarz algorithm.
\newblock In {\em NIPS}, 2014.

\bibitem{nesterov2004introductory}
Y.~Nesterov.
\newblock {\em Introductory lectures on convex optimization}, volume~87.
\newblock Springer Science \& Business Media, 2004.

\bibitem{nitanda2014stochastic}
A.~Nitanda.
\newblock Stochastic proximal gradient descent with acceleration techniques.
\newblock In {\em Advances in Neural Information Processing Systems}, pages
  1574--1582, 2014.

\bibitem{Polyak-average-sgd-1992}
B.~T. Polyak and A.~B. Juditsky.
\newblock Acceleration of stochastic approximation by averaging.
\newblock {\em SIAM J. Control and Optimization}, 30:838--855, 1992.

\bibitem{raydan1993barzilai}
M.~Raydan.
\newblock On the {B}arzilai and {B}orwein choice of steplength for the gradient
  method.
\newblock {\em IMA Journal of Numerical Analysis}, 13(3):321--326, 1993.

\bibitem{raydan1997barzilai}
M.~Raydan.
\newblock The {B}arzilai and {B}orwein gradient method for the large scale
  unconstrained minimization problem.
\newblock {\em SIAM Journal on Optimization}, 7(1):26--33, 1997.

\bibitem{Reddi2016SVRG}
S.~J. Reddi, A.~Hefny, S.~Sra, B.~Poczos, and A.~Smola.
\newblock Stochastic variance reduction for nonconvex optimization.
\newblock 2016.

\bibitem{schmidt2013minimizing}
R.~L. Roux, M.~Schmidt, and F.~Bach.
\newblock A stochastic gradient method with an exponential convergence rate for
  finite training sets.
\newblock In {\em Advances in Neural Information Processing Systems}, pages
  2663--2671, 2012.

\bibitem{schraudolph2007stochastic}
N.~N. Schraudolph, J.~Yu, and S.~G{\"u}nter.
\newblock A stochastic quasi-newton method for online convex optimization.
\newblock In {\em International Conference on Artificial Intelligence and
  Statistics}, pages 436--443, 2007.

\bibitem{Shai-Zhang-SDCA}
S.~Shalev-Shwartz and T.~Zhang.
\newblock Stochastic dual coordinate ascent methods for regularized loss
  minimization.
\newblock {\em Jornal of Machine Learning Research}, 14:567--599, 2013.

\bibitem{sopyla2015stochastic}
K.~Sopy{\l}a and P.~Drozda.
\newblock Stochastic gradient descent with {B}arzilai-{B}orwein update step for
  svm.
\newblock {\em Information Sciences}, 316:218--233, 2015.

\bibitem{Wang-Ma-2007}
Y.~Wang and S.~Ma.
\newblock Projected {B}arzilai-{B}orwein methods for large scale nonnegative
  image restorations.
\newblock {\em Inverse Problems in Science and Engineering}, 15(6):559--583,
  2007.

\bibitem{Wen-SISC-2010}
Z.~Wen, W.~Yin, D.~Goldfarb, and Y.~Zhang.
\newblock A fast algorithm for sparse reconstruction based on shrinkage,
  subspace optimization, and continuation.
\newblock {\em SIAM J. SCI. COMPUT}, 32(4):1832--1857, 2010.

\bibitem{sparsa-2009}
S.~J. Wright, R.~D. Nowak, and M.~A.~T. Figueiredo.
\newblock Sparse reconstruction by separable approximation.
\newblock {\em IEEE Transactions on Signal Processing}, 57(7):2479--2493, 2009.

\bibitem{xiao2014proximal}
L.~Xiao and T.~Zhang.
\newblock A proximal stochastic gradient method with progressive variance
  reduction.
\newblock {\em SIAM Journal on Optimization}, 24(4):2057--2075, 2014.

\bibitem{Zhang-importance-sampling-2015}
P.~Zhao and T.~Zhang.
\newblock Stochastic optimization with importance sampling for regularized loss
  minimization.
\newblock In {\em ICML}, 2015.

\end{thebibliography}
\bibliographystyle{plain}

\end{document}